\documentclass[letterpaper,12pt]{article}

\usepackage{graphicx}
\usepackage{amssymb}
\usepackage{amsthm}
\usepackage{amsmath}
\usepackage{setspace}
\usepackage[top=1in,right=1.25in,left=1.25in,bottom=1in]{geometry}

\newtheorem{theorem}{Theorem}
\newtheorem{proposition}{Proposition}
\begin{document}

\title{Diagnostic Tests for Non-causal Time Series with Infinite Variance}

\author{Yunwei Cui  \and Rongning Wu \and Thomas J. Fisher  \thanks{Corresponding author Yunwei Cui is an Assistant Professor in the Computer and Mathematical Sciences Department, University of Houston Downtown, Houston, TX 77002 (Phone: 713-226-5568; Fax: 713-221-8086; cuiy@uhd.edu); Rongning Wu is an Assistant Professor in the Zicklin School of Business at Baruch College, The City University of New York, New York, NY 10010 (rongning.wu@baruch.cuny.edu);Thomas Fisher is an Assistant Professor of Statistics at the University of Missouri-Kansas City, Kansas City, MO 64110 (fishertho@umkc.edu)}  }

\maketitle

\begin{abstract}
We study goodness-of-fit testing for non-causal autoregressive time series with non-Gaussian stable noise. To model time series exhibiting sharp spikes or occasional bursts of outlying observations, the exponent of the non-Gaussian stable variables is assumed to be less than two. Under such conditions, the innovation variables have no finite second moment. We proved that the sample autocorrelation functions of the trimmed residuals are asymptotically normal. Nonparametric tests are also investigated. The rank correlations 
of the residuals or the squared residuals are shown to be asymptotically normal.
Thus, an assortment of portmanteau statistics are available for model assessment.      

\end{abstract}




\section{Introduction} \label{intro}

Infinite variance autoregressive (AR) time series models have various practical applications. For example, Resnick (1997) fitted such a model to interarrival times between packet transmissions on a computer network, Gallagher (2001) studied differenced sea surface temperatures and fitted a symmetric $\alpha$-stable AR model, and Ling (2005) examined the daily log-returns of the Hang Seng Index in the Hong Kong stock market. When modeling infinite variance autoregressive processes, non-Gaussian $\alpha$-stable distributions (i.e. the exponent parameter $\alpha<2$) are often adopted to specify the innovation process due to their intriguing mathematical properties. This rich class of probability distributions allows heavy tails and skewness, the features exhibited in many observed time series including signal processing in electrical engineering Stuck
and Kleiner (1974); Sheng and Chen (2011), portfolio selection Rachev et al. (2004), and asset allocation Tokat and Schwartz (2002).
So, the use of $\alpha$-stable AR models is well justified both theoretically and empirically.

When studying AR processes, causality (all roots of AR polynomial are outside the unit circle) is conventionally assumed. However, such an assumption is only needed when the study is carried out within the classical Gaussian framework in order to ensure the identifiability of model parameters. Indeed, for every non-causal Gaussian AR process there exists an equivalent causal representation in the sense that the two processes have the same mean and autocorrelation functions (see Brockwell and
Davis (1991)). Since a Gaussian distribution is uniquely determined by its first two moments, the two processes necessarily possess the identical probability structure and hence are indistinguishable. In contrast, under a non-Gaussian setting, a non-causal AR process will have a different probability structure than its causal representation. In other words, for a non-Gaussian AR process the model parameters are identifiable and the model can be configured uniquely without being confined to the causal case; see Breidt and Davis (1992) and Rosenblatt (2000).

In this work we consider diagnostic tests for non-Gaussian non-causal $\alpha$-stable AR processes. We remove the assumption of causality and refer to such processes as general AR processes. There has been a certain amount of work in the literature on general AR processes. For example, Breidt et al. (1991) discussed a maximum likelihood procedure for parameter estimation for autoregressive processes with non-Gaussian innovations. Andrews et al. (2009) studied maximum likelihood estimation for general AR processes with non-Gaussian $\alpha$-stable innovations. They showed that, when fitting trading volumes of the Wal-Mart stock, a general model yielded a better description of the observed data in the sense that the residuals are more compatible with the assumption of independent innovations than the residuals produced by its causal representation. Lanne et al. (2010) considered forecasting of the non-causal AR time series and demonstrated the improvements in the change-of-direction forecasts when relaxing causality in the AR model fitted to the US inflation series. Recently Andrews and Davis (2011) developed a procedure of model identification for infinite variance AR processes and showed that minimizing Gaussian-based AIC yields a consistent estimator of the AR order.

Compared to the devotions received to parameter estimation and model identification for infinite variance non-causal AR processes, model diagnostics have not been fully addressed so far. This work intends to fill the gap. Utilizing the recent results of  
Lee and Ng (2010) and Bouhaddioui and Ghoudi (2012) we develop portmanteau test procedures for checking the goodness-of-fit of the non-causal $\alpha$-stable AR model, where the model parameters are fit using maximum likelihood estimation. As second moments do not exist for infinite variance models, the behavior of the sample autocorrelation of the residuals from the fitted model is hard to harness for the purpose of model diagnostics. To circumvent the difficulty, we propose to use the trimmed residuals or nonparametric procedures based on the ranks of the residuals or the squared residuals. We show that the sample autocorrelation of trimmed residuals at a given lag for fitted general AR processes is asymptotically normal and hence the commonly used portmanteau tests in the classical Gaussian framework that are based on sample ACF, such as Box and Pierce (1970) and Ljung and Box (1978), can be easily extended to an infinite variance setting. We also proved that the rank correlations 
of the residuals or the squared residuals are asymptotically normal. Thus nonparametric tests 
could also be developed for model diagnostic purpose. 

The rest of the paper is organized as follows. In section 2, we introduce the necessary background material to derive the asymptotic distribution of trimmed residuals. We then discuss the use of nonparametric and propose nonparametric methods.Using the asymptotic properties we propose an assortment of portmanteau test based  on the classical methods and recent results. In section 3, we examine the finite sample performance of the proposed procedures through simulation studies. We check and compare the empirical sizes and powers of the tests.  All technical proofs are relegated to the Appendix.

\section{Theoretical Results}\label{sec:results}

\subsection{Preliminaries}\label{sec:preliminaries}
Let $\{Y_t\}$ be the autoregressive process satisfying the stochastic difference equation 
\begin{equation}\label{eq:ardef}
\phi(B)Y_{t}=Z_t,
\end{equation}
where the AR characteristic polynomial has no zeros on the units circle, $\phi(z):=1-\phi_1z-\cdots-\phi_p z^p\neq 0$ for $|z|=1$, and the i.i.d innovation variables $\{Z_t\}$ have a stable distribution with exponent $\alpha\in (0, 2)$. We also assume that the AR characteristic polynomial could be written as the product of causal and purely non-causal polynomials, 
\[
\phi(z)=(1-\theta_{1}z-\cdots-\theta_{r} z^r)(1-\theta_{r+1}z-\cdots-\theta_{r+s} z^s).
\]  
Then the unique strictly stationary solution to (\ref{eq:ardef}) is given by $Y_t=\displaystyle\sum\limits_{j=-\infty}^{\infty}\psi_jZ_{t-j}$, where $\psi_j$'s are determined by the Laurent series expansion for $1/\phi(z)$, $1/\phi(z)=\displaystyle\sum\limits_{j=-\infty}^{\infty}\psi_j z^j$. It is well known that the coefficients $\{\psi_j\}$ are geometrically decaying; namely there exist $C_1>0$ and $0<D_1<1$ such that $|\psi_j|<C_1D_1^{|j|}$ for all j. Now let $\bar{\psi}_j=\psi_{-j}$, for $j>0$.  We rewrite the solution to (\ref{eq:ardef}) as 
\begin{equation}
Y_t=\displaystyle\sum\limits_{j=0}^{\infty}\psi_jZ_{t-j}+\displaystyle\sum\limits_{j=1}^{\infty}\bar{\psi}_j Z_{t+j}.
\label{e:solu}
\end{equation}

For the AR model defined in (\ref{eq:ardef}), let $\hat{\phi}$ be the MLE estimator by Andrews et al. (2009). Then $n^{1/\alpha}(\phi-\hat{\phi})\stackrel{L}{\rightarrow}S$, where $S$ is some random variable. It could be shown that there exists $\delta$, satisfying $2\alpha/(2+\alpha)<\delta<\mathrm{min}(\alpha, 1)$, such that $n^{-1/\alpha}=o(n^{-1/\delta+1/2})$.

Suppose the observed time series is represented as $\{Y_{-p+1},\cdots ,Y_{0}, Y_{1}, \cdots, Y_{n}\}$. Then the residuals of the fitted model, $\{\hat{Z_{t}}\}_{t=1}^n$,  are  given by 
\begin{equation}
\hat{Z_t}=Y_t-\hat{\phi}_1Y_{t-1}-\cdots-\hat{\phi}_{p}Y_{t-p}. 
\label{eq:residual}
\end{equation}
Let $\{\hat{Z}_t\}_{t=1}^{n}$ be the residuals of the fitted model. For some predetermined lower percentile $\lambda^L$ and upper percentile $\lambda^U$, let
$\hat{M}_n^L$ and $\hat{M}_n^U$ be the $(n\lambda^L)$-th and $(n\lambda^U)$-th order statistics of $\{\hat{Z}_t\}_{t=1}^n$, respectively. We define the following trimmed residuals
\[
\hat{\tau}_t=\hat{Z}_t I_{(\hat{M}_n^L<\hat{Z}_t<\hat{M}_n^U)}.
\]

The goal is to test the hypotheses where the null ($H_0$) is that the ARMA model (\ref{eq:ardef}) with $s>0$ is adequately identified. For the trimmed residuals, the sample autocorrelation at lag $k$, $\hat{\rho}_k$, is computed by the formula
\begin{equation}\label{eq:trimmedACF}
\hat{\rho}_k=\frac{\left(\sum_{t=k+1}^n\hat{\tau}_t\hat{\tau}_{t-k}\right)
-\left(\sum_{t=k+1}^n\hat{\tau}_t\right)\left(\sum_{t=k+1}^n\hat{\tau}_{t-k}\right)/(n-k)}
{\left(\sum_{t=1}^{n}\hat{\tau}_t^2\right)-\left(\sum_{t=1}^n\hat{\tau}_t\right)^2/n} .
\end{equation}

\begin{theorem}\label{thm:acfDistro}
If the model (\ref{eq:ardef}) is correctly identified by the MLE method, then, for any positive integer $m$, we have
\[
\sqrt{n}\hat{\pmb{\rho}}_{(m)}\stackrel{D}{\rightarrow}N(0,\mathbf{I}_m),
\]
where $\hat{\pmb{\rho}}_{(m)}:=(\hat{\rho}_1, \ldots, \hat{\rho}_m)^{T}$ and $\mathbf{I}_m$ is the $m\times m$ identity matrix
\end{theorem}

The sample partial autocorrelation (PACF) at lag $k$, $\hat{\pi}_k$, can be derived by Durbin--Levison algorithm:
\begin{equation}
\hat{\pi}_k=\frac{\hat{\rho}_k-\hat{\pmb{\rho}}_{(k-1)}^T \mathbf{R}_{(k-1)}^{-1}\hat{\pmb{\rho}}_{(k-1)}^{*}}
{1-\hat{\pmb{\pmb{\rho}}}_{(k-1)}^T \mathbf{R}_{(k-1)}^{-1}\hat{\pmb{\rho}}_{(k-1)}} , \label{e:par}
\end{equation}
where $\hat{\pmb{\rho}}_{(k-1)}=(\hat{\rho}_1, \ldots, \hat{\rho}_{k-1})^{T}$,
$\mathbf{R}_{(k-1)}=(\hat{\rho}_{|i-j|})_{i,j=1}^k$ (i.e. the symmetric Toeplitz matrix generated by $(1, \hat{\rho}_1,
\ldots, \hat{\rho}_{k-1})$), and $\hat{\pmb{\rho}}_{(k-1)}^{*}=(\hat{\rho}_{k-1}, \ldots, \hat{\rho}_1)^{T}$.

\begin{theorem}\label{thm:pacfDistro}
If the model (\ref{eq:ardef}) is correctly identified by the MLE method, then, for any positive integer $m$, we have
\[
\sqrt{n}\hat{\pmb{\pi}}_{(m)}\stackrel{D}{\rightarrow}N(0,\mathbf{I}_m),
\]
where $\hat{\pmb{\pi}}_{(m)}:=(\hat{\pi}_1, \ldots, \hat{\pi}_m)^{T}$ and $\mathbf{I}_m$ is the $m\times m$ identity matrix
\end{theorem}

Nonparametric portmanteau tests could also be developed. The following 
result provides the foundation for nonparametric tests based 
on the empirical process of the residuals or the squared residuals.  
Let $\tilde{r}_{j}=\sum_{i=1}^{n}\mathrm{I}\{\hat{Z}_i\leq \hat{Z}_j\}/n$ 
be the normalized rank of $\hat{Z}_j$  and define the rank correlation as $\hat{\gamma}_i=\sum_{t=1}^{n-i}(\tilde{r}_{t}-1/2)(\tilde{r}_{t+i}-1/2)$.
We can also define the rank correlations for the squared residuals, $\hat{\gamma}_i^{*}$,
in the same fashion.

\begin{theorem}\label{thm:rankcor}
If the model (\ref{eq:ardef}) is correctly identified 
by MLE method, then, for any positive integer $m$, we have
\begin{eqnarray*}
12\sqrt{n}\hat{\pmb{\gamma}}_{(m)}\stackrel{D}{\rightarrow}N(0,\mathbf{I}_m)\\
12\sqrt{n}\hat{\pmb{\gamma}}^{*}_{(m)}\stackrel{D}{\rightarrow}N(0,\mathbf{I}_m)
\end{eqnarray*}
where $\hat{\pmb{\gamma}}_{(m)}:=(\hat{\gamma}_1, \ldots, \hat{\gamma}_m)^{T}$,  
$\hat{\pmb{\gamma}}_{(m)}^{*}:=(\hat{\gamma}_1^{*}, \ldots, \hat{\gamma}_m^{*})^{T}$
and $\mathbf{I}_m$ is the $m\times m$ identity matrix.
\end{theorem}

\hfill \fbox \hfill

\subsection{Goodness-of-fit testing}

The results of Theorems \ref{thm:acfDistro}, \ref{thm:pacfDistro} and \ref{thm:rankcor} allow for the construction of the so called Portmanteau Statistics for time series goodness-of-fit. A Box and Pierce (1970) or Ljung
and Box (1978) type statistic can be constructed, consider:
\begin{equation}\label{eq:LjungBox}
Q_{\ell b}(m) = n(n+2)\sum_{k=1}^m \frac{\hat{\rho}_k}{n-k}.
\end{equation}
Under the null hypothesis, the Ljung Box type statistic will behave as a chi-square random variable with $m$ degrees of freedom.

A statistic inspired by Monti (1994) can be constructed utilizing the partial autocorrelation function of trimmed residuals and Theorem \ref{thm:pacfDistro},
\begin{equation}\label{eq:Monti}
Q_{mt}(m)=n(n+2)\displaystyle\sum\limits_{k=1}^{m}\frac{\hat{\pi}_k^2}{n-k}
\end{equation}
will be asymptotically distributed as a chi-square random variable with $m$ degrees of freedom for a given positive integer $m$. 

Recent work in the literature has suggested asymmetric statistics may be more powerful in some situations than the symmetric (i.e. equally weighted) Ljung Box and Monti type statistics. Define $\hat{R}_m$ as the Toeplitz matrix of autocorrelations:
\[
\hat{R}_m= \begin{bmatrix}
       1           & \hat{\rho}_1 & \ldots &   \hat{\rho}_m         \\
       \hat{\rho}_1& 1            & \ldots &   \hat{\rho}_{m-1}     \\
       \vdots      & \vdots       & \ddots &   \vdots     \\
     \hat{\rho}_m          & \hat{\rho}_{m-1} & \ldots &1
     \end{bmatrix}.
\]
Pe{\~n}a and Rodr{\'{\i}}guez (2002)
suggest a statistics based on the likelihood ratio test from multivariate analysis. Their statistic is $\hat{D}=n(1-|\hat{R}_m|^{1/m})$. Utilizing the asymptotic normality from Theorem \ref{thm:pacfDistro} and an application of the delta-method, the asymptotic distribution under the null hypothesis can be shown to satisfy
\begin{equation}\label{eq:weightDistro}
\hat{D}\stackrel{D}{\rightarrow}\displaystyle\sum\limits_{k=1}^{m}\frac{m-k+1}{m}\chi_k^{2}
\end{equation}
where each $\chi_k^2$ is a chi-square random variable with one degree of freedom. This distribution is difficult to write explicitly but can be well approximated by a Gamma distribution; see Pe{\~n}a and Rodr{\'{\i}}guez (2002) for details.

In Pe{\~n}a and Rodr{\'{\i}}guez (2006) they suggest the sum of the log of one minus the squared partial autocorrelation function. Utilizing Theorem 2, that statistic can also be shown to satisfy (\ref{eq:weightDistro}). Mahdi and McLeod (2012) generalize the result of Pe{\~n}a and Rodr{\'{\i}}guez (2002, 2006) to the multivariate time series setting. In the univariate case their statistic is
\begin{equation}\label{eq:MahdiMcLeod}
 Q_{gv}(m) = \frac{-3n}{2m+1}\log |\hat{R}_m|
\end{equation}
and the distribution follows a result similar to (\ref{eq:weightDistro}) and can be approximated with a chi-square with $(3/2)m(m + 1)/(2m + 1)$ degrees of freedom.

Recently, Fisher and Gallagher (2012) suggest an alternative asymmetric test compared to those based on the determinant of the matrix $\hat{R}_m$. They suggest a Weighted Ljung Box
\begin{equation}\label{eq:FisherGallagher}
Q_{w\ell}(m) = n(n+2)\sum_{k=1}^m \frac{m-k+1}{m} \frac{\hat{\rho}_k^2}{n-k}
\end{equation}
which is shown to satisfy the distribution in (\ref{eq:weightDistro}) and can be well approximated by a Gamma random variable with shape $\alpha = 3m(m+1)/(8m+4)$ and scale $\beta = 2(2m + 1)/3m$. Likewise, a Weighted Monti statistic is also introduced that follows the same asymptotic distribution under the null hypothesis.

\section{Simulation Studies}\label{sec:simulations}

Computation on $\alpha$-stable distributions has been well studied and is known to be computationally difficult. Our studies were performed in the GNU-licensed R-Project utilizing the stable distribution in the \verb*|stabledist| package with parameterization method zero. Due to the computational intensity in optimizing the likelihood function in Andrews et al. (2009), much of our studies were run in a parallel framework utilizing the \verb*|multicore| package. Similar to Andrews et al. (2009), when optimizing the likelihood function we generate 1200 random initial conditions; the likelihood function is found for each, and then the Nelder-Mead optimization routine is run on the best eight. The parameters for maximum likelihood function of those eight is chosen as the MLE for the general AR process with $\alpha$-stable innovations. Since the maximum likelihood function is found we can easily calculate the model identification criterion from Andrews and Davis (2011) as well.

In our studies we compare the Ljung Box type statistic in (\ref{eq:LjungBox}), the Monti type in (\ref{eq:Monti}), the Mahdi McLeod type in (\ref{eq:MahdiMcLeod}), the corresponding Weighted version of Box-Pierce $Q_{wb}$, Monti bype $Q_{wm}$, and Ljung Box test $Q_{w\ell}$ in (\ref{eq:FisherGallagher}) and the nonparametric test $Q_{rk}=144n\sum_{k=1}^m \hat{\gamma}^2_k$ for the residuals and 
$Q_{rks}=144n\sum_{k=1}^m (\hat{\gamma}^*_k)^2$ for the squared residuals.
The Mahdi McLeod was chosen over the suggestions in Pe{\~n}a and Rodr{\'{\i}}guez (2002, 2006) since it is numerically stable (see Lin and McLeod (2006)), has conservative Type I error performance and is implemented in the \verb*|portes| package. The statistics from Fisher and Gallagher (2012) are available in the \verb*|WeightedPortTest| package and include unweighted versions as well; i.e. the traditional Ljung Box and Monti types. When trimming the residuals, we truncate at the first and 99th percentiles. 

\begin{table}[tbp]
\caption{\small{Empirical sizes of non-causal AR(2) model with s=1, n=500}} \vspace{9pt}
\centering %
\begin{tabular}{cc|ccccc|ccc}
\hline 
\multicolumn{2}{c|}{AR(2)} & 
\multicolumn{5}{c|}{$$} & 
\multicolumn{3}{c}{$$}\tabularnewline 
$\phi_1=2.8$ & $\phi_2=-1.6$ & $Q_{\ell b}$ & $Q_{mt}$ & $Q_{rk}$ & $Q_{gv}$ & $Q_{rks}$ &$Q_{w b}$ & $Q_{w\ell}$ & $Q_{wm}$\tabularnewline 
\hline %
\hline %

$\alpha=1.8$& 5&0.042	&0.045	&0.030	&0.020 &0.051	&0.034	&0.036	&0.037\tabularnewline %
$\beta=0$& 10&0.050	&0.049	&0.035	&0.020	& 0.034 &0.035	&0.039	&0.040\tabularnewline %
$\gamma=1$& 15&0.045	&0.049	&0.035	&0.018	&0.039  &0.046	&0.050	&0.046\tabularnewline %
$\delta=0$& 20&0.043	&0.045	&0.034	&0.017	&0.039  &0.049	&0.056	&0.049\tabularnewline %
& 25&0.042	&0.054	&0.038	&0.014	&0.047	&0.039 &0.052	&0.048\tabularnewline \hline%

$\alpha=1.5$& 5&0.032	&0.029	&0.030	&0.031	&0.043  &0.030	&0.031	&0.034\tabularnewline %
$\beta=0$& 10&0.044	&0.046	&0.032	&0.023	&0.041 &0.032	&0.034	&0.036\tabularnewline %
$\gamma=1$& 15&0.045	&0.044	&0.024	&0.020	&0.038  &0.036	&0.039	&0.042\tabularnewline %
$\delta=0$& 20&0.048	&0.048	&0.039	&0.016	&0.041 &0.038	&0.042	&0.042\tabularnewline %
& 25&0.039	&0.046	&0.037	&0.015	&0.046 &0.034	&0.043	&0.047\tabularnewline \hline%

$\alpha=1.2$& 5 &0.059	&0.062	&0.059	&0.037 &0.059	&0.052	&0.050	&0.053\tabularnewline %
$\beta=0$ & 10&0.060	&0.056	&0.062	&0.029 & 0.062	&0.057	&0.061	&0.056\tabularnewline %
$\gamma=1$ & 15 &0.056	&0.055	&0.059	&0.025	& 0.052 &0.057	&0.062	&0.058\tabularnewline %
$\delta=0$& 20 &0.049	&0.053	&0.067	&0.021	&0.057  &0.057	&0.060	&0.056\tabularnewline %
    & 25&0.050	&0.062	&0.061	&0.018	&0.059 &0.054	&0.059	&0.059\tabularnewline  \hline%

$\alpha=0.8$& 5&0.088	&0.083	&0.073	&0.068	&0.087  &0.084	&0.087	&0.086\tabularnewline %
$\beta=0$ & 10 &0.078	&0.081	&0.085	&0.053	&0.069 &0.088	&0.091	&0.088\tabularnewline %
$\gamma=1$ &15 &0.081	&0.084	&0.078	&0.047	&0.071 &0.087	&0.091	&0.092\tabularnewline %
$\delta=0$& 20 &0.086	&0.080	&0.074	&0.048	&0.068  &0.079	&0.085	&0.084\tabularnewline %
    & 25&0.083	&0.078	&0.071	&0.046	      &0.063 &0.074	&0.082	&0.080\tabularnewline  \hline
    
$\alpha=1.8$&5  &0.045	&0.045	&0.038	&0.022	& 0.041&0.030	&0.030	&0.029\tabularnewline %
$\beta=0.5$ &10 &0.046	&0.047	&0.034	&0.017	& 0.039 &0.043	&0.044	&0.044\tabularnewline %
$\gamma=1$  &15 &0.042	&0.040	&0.040	&0.019	& 0.040 &0.041	&0.043	&0.042\tabularnewline %
$\delta=0$  &20 &0.046	&0.039	&0.037	&0.014	& 0.046 &0.037	&0.042	&0.041\tabularnewline %
            &25 &0.045	&0.044	&0.035	&0.015	& 0.044   &0.033	&0.040	&0.039\tabularnewline \hline

$\alpha=1.5$&5 &0.048 &0.043 &0.037  &0.036 &0.044 &	0.039&	0.041&	0.042\tabularnewline %
$\beta=0.5$ &10 &0.038 &0.043 &0.031 &0.032 &0.033  &	0.034&	0.038&	0.037\tabularnewline %
$\gamma=1$  &15 &0.042 &0.040 &0.042 &0.034 &0.035  &	0.038&	0.042&	0.034\tabularnewline %
$\delta=0$  &20 &0.036 &0.044 &0.035 &0.034 & 0.034 &	0.035&	0.036&	0.036\tabularnewline %
            &25&0.035&0.033&	0.030  &0.034 &0.036  & 0.031&	0.037&	0.036\tabularnewline \hline
$\alpha=1.2$&5 &0.066	&0.063	&0.047	&0.054&0.069	&0.062	&0.064	&0.060\tabularnewline %
$\beta=0.5$ &10&0.061	&0.066	&0.045	&0.066&0.064	&0.062	&0.064	&0.061\tabularnewline %
$\gamma=1$  &15&0.063	&0.070	&0.040	&0.070&0.064	&0.064	&0.069	&0.063\tabularnewline %
$\delta=0$  &20&0.06	&0.075	&0.053	&0.075&0.062	&0.065	&0.067	&0.065\tabularnewline %
            &25&0.06	&0.061	&0.054	&0.076&0.056	&0.065	&0.068	&0.064\tabularnewline \hline
$\alpha=0.8$&5&0.099	&0.104	&0.098	&0.043	&0.094 &0.099	&0.100	&0.102\tabularnewline %
$\beta=0.5$&10&0.114	&0.113	&0.091	&0.047	&0.083  &0.115	&0.117	&0.115\tabularnewline %
$\gamma=1$&15&0.115	&0.110	&0.098	&0.053	&0.079 &0.104	&0.109	&0.115\tabularnewline %
$\delta=0$ &20&0.105	&0.102	&0.094	&0.059	& 0.073 &0.103	&0.110	&0.109\tabularnewline %
	&25&0.108&0.115	&0.087	&0.060	&0.071  &0.095	&0.105	&0.109\tabularnewline \hline             
\end{tabular}
\label{sizeempone}
\end{table}

We check the finite sample sizes and powers of the proposed tests for different AR(1)
and AR(2) models. For each selected model the simulation is run $1000$ times. The 
results are summarized in Table $1$ through Table $3$. 
Overall, these tests perform well when $\alpha\ge 1.5$. No test dominate the performance.  
As $\alpha$ decreases the 
empirical sizes increase. But since in practice the $\alpha$ for the fitted model 
is always above $1.5$, the problem does not cause big concerns to us. 

\begin{table}[tbp]
\caption{\small{Empirical powers for non-causal AR(2) model fitted as non-causal AR(1), n=50}} \vspace{9pt}
\centering %
\begin{tabular}{cc|ccccc}
\hline 
\multicolumn{2}{c|}{AR(2)} & 
\multicolumn{5}{c}{} \tabularnewline 
$\phi_1=-1.2$ & $\phi_2=1.6$ & $Q_{\ell b}$ & $Q_{mt}$ & $Q_{gv}$& $Q_{w b}$ & $Q_{rk}$  \tabularnewline 
\hline %
\hline %

$\alpha=0.8$ &5&0.616	&0.541	&0.67	&0.756	&0.915	\tabularnewline %
$\beta=0$   &10&0.483	&0.325	&0.523	&0.629	&0.833	\tabularnewline %
$\gamma=1$  &15&0.397	&0.246	&0.420	&0.574	&0.752	\tabularnewline %
$\delta=0$  &20&0.373	&0.167	&0.326	&0.538	&0.693	\tabularnewline %
            &25&0.356	&0.135	&0.258	&0.504	&0.623	\tabularnewline \hline%

$\alpha=1.2$ &5&0.533	&0.454	&0.550	&0.633	&0.684	\tabularnewline %
$\beta=0$   &10&0.425	&0.311	&0.446	&0.556	&0.550	\tabularnewline %
$\gamma=1$  &15&0.381	&0.219	&0.349	&0.509	&0.462	\tabularnewline %
$\delta=0$  &20&0.376	&0.167	&0.276	&0.487	&0.404	\tabularnewline %
            &25&0.349	&0.117	&0.204	&0.460	&0.360	\tabularnewline \hline%
            
$\alpha=1.5$ &5&0.426	&0.358	&0.428	&0.497	&0.456	\tabularnewline %
$\beta=0$   &10&0.323	&0.241	&0.335	&0.447	&0.343	\tabularnewline %
$\gamma=1$  &15&0.302	&0.174	&0.262	&0.407	&0.264	\tabularnewline %
$\delta=0$  &20&0.284	&0.138	&0.192	&0.379	&0.227	\tabularnewline %
            &25&0.311	&0.107	&0.145	&0.361	&0.178	\tabularnewline \hline%
            
$\alpha=1.8$	&5	&0.313	&0.263	&0.305	&0.362	&0.323	\tabularnewline	%
$\beta=0$	&10	&0.235	&0.188	&0.247	&0.319	&0.214	\tabularnewline	%
$\gamma=1$	&15	&0.228	&0.145	&0.193	&0.296	&0.155	\tabularnewline	%
$\delta=0$	&20	&0.230	&0.115	&0.142	&0.275	&0.116	\tabularnewline	%
	&25	&0.246	&0.092	&0.094	&0.282	&0.095	\tabularnewline	\hline%

\end{tabular}
\end{table}


\begin{table}[tbp]
\caption{\small{Empirical powers for non-causal AR(2) model fitted as non-causal AR(1), n=75}} \vspace{9pt}
\centering %
\begin{tabular}{cc|ccccc}
\hline 
\multicolumn{2}{c|}{AR(2)} & 
\multicolumn{5}{c}{} \tabularnewline 
$\phi_1=-1.2$ & $\phi_2=1.6$ & $Q_{\ell b}$ & $Q_{mt}$ & $Q_{gv}$& $Q_{w b}$ & $Q_{rk}$  \tabularnewline 
\hline %
\hline %

$\alpha=0.8$ &5		&0.909	&0.851	&0.933	&0.948	&0.99	\tabularnewline	%
$\beta=0$	   &10	&0.692	&0.591	&0.853	&0.905	&0.973	\tabularnewline	%
$\gamma=1$	 &15	&0.601	&0.447	&0.723	&0.801	&0.954	\tabularnewline	%
$\delta=0$	 &20	&0.554	&0.351	&0.617	&0.727	&0.935	\tabularnewline	%
	           &25	&0.492	&0.296	&0.547	&0.687	&0.900	\tabularnewline	\hline%

$\alpha=1.2$ &5		&0.831	&0.791	&0.868	&0.873	&0.911	\tabularnewline	%
$\beta=0$	   &10	&0.678	&0.583	&0.783	&0.833	&0.829	\tabularnewline	%
$\gamma=1$	 &15	&0.598	&0.470	&0.694	&0.776	&0.770	\tabularnewline	%
$\delta=0$	 &20	&0.568	&0.377	&0.608	&0.732	&0.729	\tabularnewline	%
	           &25	&0.533	&0.312	&0.537	&0.694	&0.673	\tabularnewline	\hline%

$\alpha=1.5$ &5		&0.727	&0.692	&0.744	&0.772	&0.785	\tabularnewline	%
$\beta=0$	   &10	&0.606	&0.528	&0.689	&0.742	&0.657	\tabularnewline	%
$\gamma=1$	 &15	&0.537	&0.422	&0.617	&0.687	&0.580	\tabularnewline	%
$\delta=0$	 &20	&0.506	&0.344	&0.541	&0.658	&0.522	\tabularnewline	%
	           &25	&0.492	&0.277	&0.469	&0.635	&0.464	\tabularnewline	\hline%

$\alpha=1.8$ &5		&0.554	&0.519	&0.586	&0.618	&0.628	\tabularnewline	%
$\beta=0$	   &10	&0.470	&0.420	&0.521	&0.586	&0.485	\tabularnewline	%
$\gamma=1$	 &15	&0.421	&0.339	&0.470	&0.545	&0.394	\tabularnewline	%
$\delta=0$	 &20	&0.413	&0.285	&0.408	&0.522	&0.343	\tabularnewline	%
	           &25	&0.389	&0.240	&0.343	&0.507	&0.288	\tabularnewline	\hline%
          
\end{tabular}
\end{table}

\section{Appendix}
To prove Theorem 1, we follow the method used in Lee and Ng (2010). 
Since Proposition $5.2$ is true for the innovation process in general, 
we can use it for free. The key is to establish the remaining technical 
lemmas in their paper for the non-causal model. In the following, 
Proposition $1$ and $2$ are corresponding to Proposition 
$5.1$ and $5.3$ of Lee and Ng respectively. 

Let $\varphi_t=Z_t-\hat{Z}_t$, for $t=1, \cdots, n$. From (\ref{e:solu}) we get
\begin{equation}
\varphi_t=\displaystyle\sum\limits_{j=1}^{p}(\phi_j-\hat{\phi}_j)Y_{t-j}=\displaystyle\sum\limits_{j=1}^{p}(\phi_j-\hat{\phi}_j)\displaystyle\sum\limits_{k=0}^{\infty}\psi_kZ_{t-j-k}+\displaystyle\sum\limits_{j=1}^{p}(\phi_j-\hat{\phi}_j)\displaystyle\sum\limits_{k=1}^{\infty}\bar{\psi}_k Z_{t-j+k}
\label{e:varp}
\end{equation}
  
\noindent By changing the order of summation

\begin{eqnarray}
|\displaystyle\sum\limits_{j=1}^{p}(\phi_j-\hat{\phi}_j)\displaystyle\sum\limits_{k=0}^{\infty}\psi_kZ_{t-j-k}|
&\leq& |\displaystyle\sum\limits_{j=1}^{\infty}\displaystyle\sum\limits_{k=1}^{\mathrm{min}(j, p)}(\phi_k-\hat{\phi}_k)\psi_{j-k} Z_{t-j}|\nonumber\\
&\leq &\displaystyle\sum\limits_{j=1}^{\infty}\displaystyle\sum\limits_{k=1}^{\mathrm{min}(j, p)}\left\|\phi-\hat{\phi}\right\||\psi_{j-k}||Z_{t-j}|, 
\label{e:fst}
\end{eqnarray}

\noindent and
\begin{flalign}
|\displaystyle\sum\limits_{j=1}^{p}&(\phi_j-\hat{\phi}_j)\displaystyle\sum\limits_{k=1}^{\infty}\bar{\psi}_k Z_{t-j+k}|\nonumber\\
\leq&\displaystyle\sum\limits_{j=0}^{\infty}|\displaystyle\sum\limits_{k=1}^{p}(\phi_k-\hat{\phi}_k)\bar{\psi}_{j+k}Z_{t+j}|
+|\displaystyle\sum\limits_{j=1}^{p-1}\displaystyle\sum\limits_{k=j+1}^{p}(\phi_k-\hat{\phi}_k)\bar{\psi}_{k-j}Z_{t-j}|\nonumber\\
&\leq \displaystyle\sum\limits_{j=0}^{\infty}\displaystyle\sum\limits_{k=1}^{p}\left\|\phi-\hat{\phi}\right\||\bar{\psi}_{j+k}Z_{t+j}|
+\displaystyle\sum\limits_{j=1}^{p-1}\displaystyle\sum\limits_{k=j+1}^{p}\left\|\phi-\hat{\phi}\right\| |\bar{\psi}_{k-j}Z_{t-j}|,
\label{e:ftt}
\end{flalign}

\noindent where $\left\|\phi-\hat{\phi}\right\|$ is the Euclidean distance of $\phi$ and $\hat{\phi}$. 
By Andrews and Davis (2009) the MLE estimator of the AR polynomial coefficients,  $\hat{\phi}$, converges to 
some random variable in distribution $n^{1/\alpha}(\hat{\phi}-\phi)\stackrel{D}{\rightarrow}S$. By our assumption
that $0<\alpha<2$, we can find a $\delta$ with $2\alpha/(\alpha+2)<\delta<min\{\alpha, 1\}$ such that 
$n^{-\alpha}=o(n^{-1/\delta+1/2})$.  
Note that for any given $\epsilon>0$ there always exists a 
$\gamma_1>0$ such that $\mathrm{P}(|S|>\gamma_1)<\epsilon/2$. If 
we define define $A_n=\left\{\left\|\phi-\hat{\phi}\right\|<\gamma_1 n^{-1/\alpha}\right\}$
then there exists  $N>0$ such that
$P(A_n)>1-\epsilon$ whenever $n>N_1$.  
Under the condition of $A_n$, we can obtain an upper bound for  (\ref{e:fst})  
\begin{equation}
|\displaystyle\sum\limits_{j=1}^{p}(\phi_j-\hat{\phi}_j)\displaystyle\sum\limits_{j=0}^{\infty}\psi_kZ_{t-k}|
\leq  \gamma_1 n^{-1/\alpha}  \displaystyle\sum\limits_{j=1}^{\infty}
\displaystyle\sum\limits_{k=1}^{\mathrm{min}(j,p)}|\psi_{j-k}|
|Z_{t-j}|,
\label{e:atst}
\end{equation}

\noindent and an upper bound for  (\ref{e:ftt})

\begin{equation}
|\displaystyle\sum\limits_{j=1}^{p}
(\phi_j-\hat{\phi}_j)
\displaystyle\sum\limits_{k=1}^{\infty}\bar{\psi}_k Z_{t+k}|
\leq\gamma_1 n^{-1/\alpha}  \displaystyle\sum\limits_{j=0}^{\infty}\displaystyle\sum\limits_{k=1}^{p}|\bar{\psi}_{j+k}||Z_{t+j}|
+\gamma_1 n^{-1/\alpha}\displaystyle\sum\limits_{j=1}^{p-1}\displaystyle\sum\limits_{k=j+1}^{p} |\bar{\psi}_{k-j}||Z_{t-j}|.
\label{e:att}
\end{equation} 

Let $\psi_{j}^*=\displaystyle\sum\limits_{k=1}^{\mathrm{min}(j,p)}|\psi_{j-k}|$,
$\bar{\psi}_{j}^*=\displaystyle\sum\limits_{k=1}^{p}|\bar{\psi}_{j+k}|$, 
and $\bar{\psi}_{1,\cdots,p}=\displaystyle\sum\limits_{k=1}^{p}|\bar{\psi}_k|$.

Assuming $A_n$ is true, an upper bound for $|\varphi_t|$ is given by 

\begin{equation}
|\varphi_t|\leq \gamma_1 n^{-1/\alpha}  \displaystyle\sum\limits_{j=1}^{\infty}\psi_{j}^*|Z_{t-j}|
+\gamma_1 n^{-1/\alpha}  \displaystyle\sum\limits_{j=0}^{\infty}\bar{\psi}_{j}^*|Z_{t+j}|
+\gamma_1 n^{-1/\alpha}\displaystyle\sum\limits_{j=1}^{p-1}
\bar{\psi}_{1,\cdots,p}
|Z_{t-j}|.
\label{e:bdd}
\end{equation}

\begin{proposition}\label{proposition1}
For (\ref{e:bdd}), the following are true, 
\begin{eqnarray*}
E|\displaystyle\sum\limits_{j=1}^{\infty}\psi_{j}^*|Z_{t-j}||^\delta<\infty,\\
E|\displaystyle\sum\limits_{j=0}^{\infty}\bar{\psi}_{j}^*|Z_{t+j}||^\delta<\infty,\\
\gamma_1^{\delta} n^{-\delta/\alpha}\displaystyle\sum\limits_{t=1}^{n}\{E|  \displaystyle\sum\limits_{j=1}^{\infty}\psi_{j}^*|Z_{t-j}||^{\delta}
+E|\displaystyle\sum\limits_{j=0}^{\infty}\bar{\psi}_{j}^*|Z_{t+j}||^{\delta}\}=o(n). 
\end{eqnarray*}

\end{proposition}
\begin{proof}

The coefficients $\{\psi_j\}$ and $\{\bar{\psi_j}\}$ are geometrically decaying as $j\rightarrow\infty$. 
As a result, $\displaystyle\sum\limits_{j=1}^{\infty}|\psi_{j}|^{\delta}<\infty$
and $\displaystyle\sum\limits_{j=1}^{\infty}|\psi_{j}^*|^{\delta}<\infty$.

Change the order of summation and apply the triangle inequality, then we have

\[
\displaystyle\sum\limits_{j=1}^{\infty}|\psi_{j}^*|^{\delta}
\le \displaystyle\sum\limits_{j=1}^{\infty} \displaystyle\sum\limits_{k=1}^{\mathrm{min}(j,p)}|\psi_{j-k}|^{\delta}
=p\displaystyle\sum\limits_{j=0}^{\infty}|\psi_{j}|^{\delta}<\infty,
\]

\noindent and

\[
\displaystyle\sum\limits_{j=0}^{\infty}|\bar{\psi_{j}}^*|^{\delta}
\le \displaystyle\sum\limits_{j=0}^{\infty} \displaystyle\sum\limits_{k=1}^{p}|\bar{\psi}_{j+k}|^{\delta}
\leq p \displaystyle\sum\limits_{j=1}^{\infty}|\bar{\psi_{j}}|^{\delta}<\infty.
\]

Also by the triangle inequality (for example, page 537, Brockwell and Davis, 1991) and 
$E|Z_{t-j}|^{\delta}<\infty$
\begin{eqnarray*}
E|\displaystyle\sum\limits_{j=1}^{\infty}\psi_{j}^*|Z_{t-j}||^\delta
&\leq& E|\displaystyle\sum\limits_{j=1}^{\infty}|\psi_{j}^*|^{\delta}||Z_{t-j}|^{\delta}
<\infty,\\
E|\displaystyle\sum\limits_{j=0}^{\infty}\bar{\psi}_{j}^*|Z_{t+j}||^\delta
&\leq& E|\displaystyle\sum\limits_{j=0}^{\infty}|\bar{\psi}_{j}^*|^{\delta}||Z_{t+j}|^{\delta}
<\infty.\\
\end{eqnarray*}
\end{proof}

Given a fixed number $0<\lambda<1$ and $\beta_n$ a predetermined sequence of real numbers,
let $\chi_t=Z_t-Z_{([n\lambda])}-\beta_n$.
The Proposition 5.3. of Lee and Ng (2010) is also true for the non-causal AR sequences. 

\begin{proposition}\label{proposition2}
For any $\gamma_2>0$, 
\[
\mathrm{P}\{n^{-1/2}\displaystyle\sum\limits_{t=1}^{n}1_{(|\varphi_t|>|\chi_t|)}1_{A_n}>\gamma_2\}\rightarrow 0.
\]
\end{proposition}

\begin{proof}
As in Lee and Ng (2010), we can pick a constant $\gamma_3>0$ such that 
$\mathrm{P}(|Z_{s[n\lambda]}|>\gamma_3)$ is arbitrarily small in which 
$s(k)=j$ if $Z_j$ is the $k^{th}$ largest number among $\{Z_1,\ldots, Z_n\}$. 
To show the 
result it is sufficient to get   
\[
\displaystyle\sum\limits_{t=1}^{n}\mathrm{P}\{(|\varphi_t|>|\chi_t|)\cap{A_n}\cap(|Z_{s[n\lambda]}|<\gamma_3)\}=o(n^{1/2}).
\]

By Lee and Ng (2010), for any $t\in\{1, \ldots, n\}$,  
\[
\mathrm{P}\{(|\varphi_t|>|\chi_t|)\cap{A_n}\cap(|Z_{s[n\lambda]}|<\gamma_3)\}\leq
\frac{1}{n}+\frac{n-1}{n}E\{|\varphi_t|^{\delta}|\chi_t|^{-\delta}1_{(|Z_{s[n\lambda]}|<\gamma_3)}1_{A_n}\bigg |t\ne s([n\lambda])\}.
\]

Use triangle inequality and (\ref{e:bdd})
\begin{flalign}
\displaystyle\sum\limits_{t=1}^{n} E\{|\varphi_t|^{\delta}&|\chi_t|^{-\delta}1_{(|Z_{s[n\lambda]}|<\gamma_3)}
1_{A_n} \bigg| t\ne s([n\lambda])\}\nonumber\\
\leq 
\displaystyle\sum\limits_{t=1}^{n}
E\{&\gamma_1^{\delta} n^{-\delta/\alpha} ( \displaystyle\sum\limits_{j=1}^{\infty}\psi_{j}^*|Z_{t-j}|)^{\delta}
|\chi_t|^{-\delta}1_{(|Z_{s[n\lambda]}|<\gamma_3)}\bigg| s([n\lambda])\ne t \}+\label{e:iqt1}\\
\displaystyle\sum\limits_{t=1}^{n}
E&\{\gamma_1^\delta n^{-\delta/\alpha}  (\displaystyle\sum\limits_{j=0}^{\infty}\bar{\psi}_{j}^*|Z_{t+j}|)^{\delta}
|\chi_t|^{-\delta}1_{(|Z_{s[n\lambda]}|<\gamma_3)}\bigg| s([n\lambda])\ne t \}+\label{e:iqt2}\\
\displaystyle\sum\limits_{t=1}^{n}
&E\{\gamma_1^\delta n^{-\delta/\alpha} (\displaystyle\sum\limits_{j=1}^{p-1}\bar{\psi}_{1,\cdots,p}|Z_{t-j}|)^\delta
|\chi_t|^{-\delta}1_{(|Z_{s[n\lambda]}|<\gamma_3)}\bigg| s([n\lambda])\ne t \}.\label{e:iqt3}
\end{flalign}

To finish the proof, in the next we will show (\ref{e:iqt1}), (\ref{e:iqt2}), and (\ref{e:iqt3}) are $o(n^{1/2})$.

Conditional on $s([n\lambda])=t-j$ and $s([n\lambda])\neq t-j$, 
(\ref{e:iqt1}) is bounded above by 
\begin{flalign}
\displaystyle\sum\limits_{t=1}^{n}
\gamma_1^{\delta}\gamma_3^{\delta} & n^{-\delta/\alpha}\displaystyle\sum\limits_{j=1}^{\infty}|\psi_{j}^*|^{\delta}
E\{
|\chi_t|^{-\delta}\bigg| s([n\lambda])\ne t \}+\label{e:fbf1}\\
\displaystyle\sum\limits_{t=1}^{n}
& E\{\gamma_1^{\delta} n^{-\delta/\alpha} \displaystyle\sum\limits_{j=1}^{\infty}|\psi_{j}^*|^\delta |Z_{t-j}|^{\delta}
|\chi_t|^{-\delta}1_{(s[n\lambda]\neq t-j)}\bigg| s([n\lambda])\ne t \}.
\label{e:fbf2}
\end{flalign}

\noindent
We apply  Proposition $1$, (5.23) in Proposition 5.3 of Lee and Ng (2010), and the fact that $n^{-1/\alpha}=o(n^{-1/\delta+1/2})$ 
to (\ref{e:fbf1}) and get  
\[\displaystyle\sum\limits_{t=1}^{n}\gamma_1^{\delta}\gamma_3^{\delta}  n^{-\delta/\alpha}\displaystyle\sum\limits_{j=1}^{\infty}|\psi_{j}^*|^{\delta}
E\{
|\chi_t|^{-\delta}\bigg| s([n\lambda])\ne t \}=o(n^{1/2}).\]

\noindent For (\ref{e:fbf2}), 
two cases, $1\leq j \leq t-1$ and $j\geq t$, are considered respectively. 
When $1\leq j \leq t-1$, 
\begin{flalign*}
\displaystyle\sum\limits_{t=1}^{n}
E\{\gamma_1^{\delta} &n^{-\delta/\alpha} \displaystyle\sum\limits_{j=1}^{t-1}|\psi_{j}^*|^\delta |Z_{t-j}|^{\delta}
|\chi_t|^{-\delta}1_{(s[n\lambda]\neq t-j)}\bigg| s([n\lambda])\ne t \}\leq \\
&\displaystyle\sum\limits_{t=1}^{n}\gamma_1^{\delta} n^{-\delta/\alpha}\frac{n-2}{n-1}
E\{ \displaystyle\sum\limits_{j=1}^{t-1}|\psi_{j}^*|^\delta |Z_{t-j}|^{\delta}
|\chi_t|^{-\delta} \bigg| s([n\lambda])\ne t, t-j \}=o\left(n^{1/2}\right),
\end{flalign*}

\noindent since by (5.22) of Lee and Ng (2010) $E\{ \displaystyle\sum\limits_{j=1}^{t-1}|\psi_{j}^*|^{\delta}|Z_{t-j}|^{\delta}|\chi_t|^{-\delta}\bigg| s([n\lambda])\ne t,t-j \}=O(1)$.
When $j\geq t$, $Z_{t-j}$ is in the set $\{Z_{0}, Z_{-1}, \ldots\}$. 
Hence,  $Z_{t-j}$ is independent of $s([n\lambda])$,
which implies that 
\[E\{| \psi_{j}^*|^{\delta}|Z_{t-j}|^{\delta}|\chi_t|^{-\delta}1_{(s[n\lambda]\neq t-j)}\bigg| s([n\lambda])\ne t\}=|\psi_{j}^*|^{\delta}
 E\{|Z_{t-j}|^{\delta}\}E\{|\chi_t|^{-\delta}1_{(s[n\lambda]\neq t-j)}\bigg| s([n\lambda])\ne t \}.\]
\noindent
Now use Proposition 1 and (5.23) of Lee and Ng again
\begin{flalign*}
\displaystyle\sum\limits_{t=1}^{n}
E\{\gamma_1^{\delta} &n^{-\delta/\alpha} \displaystyle\sum\limits_{j=t}^{\infty}|\psi_{j}^*|^\delta |Z_{t-j}|^{\delta}
|\chi_t|^{-\delta}1_{(s[n\lambda]\neq t-j)}\bigg| s([n\lambda])\ne t \}\leq\\
&\displaystyle\sum\limits_{t=1}^{n}\gamma_1^{\delta} n^{-\delta/\alpha}\frac{n-1}{n}
\displaystyle\sum\limits_{j=t}^{\infty}|\psi_{j}^*|^\delta E\{|Z_{t-j}|^{\delta}\}
E\{|\chi_t|^{-\delta} 1_{(s[n\lambda]\neq t-j)} \bigg| s([n\lambda])\ne t\}=o\left(n^{1/2}\right).
\end{flalign*}

\noindent Therefore, (\ref{e:iqt1}) is $o(n^{1/2})$ .
In the same way, the result also holds for (\ref{e:iqt2}) and (\ref{e:iqt3}).

\end{proof}

\begin{proof} of Theorem \ref{thm:acfDistro}.

Let $q^L$ and $q^U$ be the $(\lambda^L)$-th and $(\lambda^U)$-th 
quantiles of $Z_t$. 
Denote the mean and standard deviation of the trimmed 
random variable $Z_tI(q^L<Z_t<q^U)$ by $\mu$ and $\sigma$, 
\[
\mu=E[Z_tI(q^L<Z_t<q^U)] \quad\textrm{and }\quad\sigma^2=Var[Z_tI(q^L<Z_t<q^U)].
\]
\noindent Let $Z_t^\mu=Z_tI_t-\mu$, then directly from Lemma $4.1$ in Lee and Ng~(2010), 
\begin{flalign}
&n^{-1/2}\left\{\displaystyle\sum\limits_{t=k+1}^{n}Z_t^\mu Z_{t-k}^\mu\right\}_{k=1,2,\ldots,m}
\stackrel{D}{\rightarrow}N(0, \sigma^4I_m),\nonumber\\
&n^{-1/2}\displaystyle\sum\limits_{t=1}^{n}Z_t^\mu
\stackrel{D}{\rightarrow}N(0, \kappa^2),\nonumber\\
&n^{-1}\displaystyle\sum\limits_{t=1}^{n}(Z_t^\mu)^2
\stackrel{p}{\rightarrow}\sigma^2,
\label{e:ori}
\end{flalign}

\noindent with $\kappa$ being certain constant associated with the distribution of $Z_t$,
and $(q^L, q^U)$.  

Now let
$M_n^L$ and $M_n^U$ be the $(n\lambda^L)$-th and $(n\lambda^U)$-th order statistics of
$\{Z_t\}_{t=1}^n$ and define $\hat{Z}_t^\mu=Z_tI_t-\mu$ and
\[
I_t=
\left\{ 
  \begin{array}{l l}
    1, &\quad \textrm{if $ M_n^L<Z_t <M_n^U$},\\
    0, & \quad \textrm{otherwise,}\\
   \end{array} \right. \quad\quad\
\hat{I_t}=
\left\{ 
  \begin{array}{l l}
    1, &\quad \textrm{if $\hat{M}_n^L<\hat{Z}_t <\hat{M}_n^U$},\\
    0, & \quad \textrm{otherwise.}\\
  \end{array} \right.   
\]

\noindent It follows from Proposition (\ref{proposition1}),
Proposition (\ref{proposition2}), and the proof of Lemma 4.2 in Lee and Ng~(2010) 
that 

\begin{flalign}
&n^{-1/2}\sum\limits_{t=k+1}^{n}|Z_t^{\mu}Z_{t-k}^{\mu}-\hat{Z}_t^{\mu}\hat{Z}_{t-k}^{\mu}|\stackrel{p}{\rightarrow}0, \quad \textrm{for $k=1,2,\ldots,m$,} \nonumber\\
&n^{-1/2}\sum\limits_{t=1}^{n}|Z_tI_t-\hat{Z}_t\hat{I}_{t}|\stackrel{p}{\rightarrow}0,\nonumber\\
&n^{-1}\sum\limits_{t=1}^{n}|(Z_t^{\mu})^2-(\hat{Z}_t^{\mu})^2|\stackrel{p}{\rightarrow}0.
\label{e:orihat}
\end{flalign}

Now note that  

\[
\sqrt{n-k}\hat{\rho}_k=\frac{\frac{1}{\sqrt{n-k}}\displaystyle\sum\limits_{t=k+1}^{n}\hat{Z}_t^\mu \hat{Z}_{t-k}^\mu-\frac{1}{\sqrt{n-k}}\displaystyle\sum\limits_{t=k+1}^{n}\hat{Z}_t^\mu
\frac{1}{n-k}\displaystyle\sum\limits_{t=k+1}^{n}\hat{Z}_{t-k}^\mu}
{\frac{1}{n}
\displaystyle\sum\limits_{t=1}^{n}(\hat{Z}_t^\mu)^2-\frac{1}{n^2}
(
\displaystyle\sum\limits_{t=1}^{n}\hat{Z}_t^\mu 
)^2
}.
\]

\noindent Combining (\ref{e:ori}) and (\ref{e:orihat}) yields the 
result. 

\end{proof}

\begin{proof} of Theorem \ref{thm:pacfDistro}.

By Theorem \ref{thm:acfDistro} and equation (\ref{e:par}).

\end{proof}

\begin{proof}of Theorem \ref{thm:rankcor}

Define the empirical copula of the residuals be defines as 
\[
C_{m,n}(u_1, \ldots, u_m)=\frac{1}{\sqrt{n-m+1}}\sum_{i=1}^{n-m+1}\left[\prod_{j=1}^{m}\mathrm{I}(\tilde{r}_{i+j-1}\leq u_j)-\prod_{j=1}^{m}u_j\right],
\]

\noindent where $(u_1, \ldots, u_m)\in [0,1]^m$. 
Cline and Brockwell (1985) showed 

\begin{equation}
\lim_ {t\rightarrow\infty}\frac{P[|Y_1|>t]}{P[|Z_1|>t]}=\sum_{j=-\infty}^{\infty}|\psi_j|^{\alpha}.
\label{e:cline1}
\end{equation}
\noindent As a result
\begin{equation}
\lim_ {t\rightarrow\infty}nP[|Y_1|>a_n t]=\sum_{j=-\infty}^{\infty}|\psi_j|^{\alpha}t^{-\alpha},
\label{e:cline2}
\end{equation}
 
\noindent for all $t>0$, where $a_n=inf\{t: nP[|Z_1|>t]\le1\}$. Then we can 
follow the same lines of Theorem 3.4 and related technical Lemmas
in Bouhaddioui and Ghoudi (2012)  
to prove that the empirical copula of the residuals $C_{m,n}$ converges to a 
continuous process $\tilde{\mathcal{C}}$ if the model 
(\ref{eq:ardef}) is correctly identified 
by MLE method. 
The continuous process $\tilde{\mathcal{C}}$ is the limit of the 
sequential empirical process of a sequence of i.i.d random variables 
identified in Genest and R{\'e}millard (2004), for which there is 
no simple expression. 
However, as in Proposition 2.1 of Genest and R\'{e}millard (2004), 
the M\"{o}bius transformation of $C_{m,n}$, $\mathcal{M}$, leads to some simple results. 
Let $A$ be a subset of $\{1, \ldots, m\}$ with $|A|>1$, the M\"{o}bius 
transformation of $C_{m,n}$ indexed by $A$ is 
\[
\mathcal{M}_A(C_{m,n})=\frac{1}{\sqrt{n-m+1}}\sum_{i=1}^{n-m+1}\prod_{j\in A}\left[\mathrm{I}(\tilde{r}_{i+j-1}\leq u_j)-u_j\right].
\]  
\noindent Then $\mathcal{M}_A(C_{m,n})$ converge jointly to continuous 
centered Gaussian processes $\mathcal{M}_A(\tilde{\mathcal{C}})$ and 
furthermore $\mathcal{M}_A(\tilde{\mathcal{C}})$ and 
$\mathcal{M}_{A^{'}}(\tilde{\mathcal{C}})$ are asymptotically independent 
whenever two sets $A\ne A^{'}$. 
Letting $A=\{1,k+1\}$, then the serial rank correlation $\hat{\gamma}_i$ 
could be derived, as in Bouhaddioui and Ghoudi (2012), 
from the  M\"{o}bius transformation of $C_{m,n}$ through

\begin{eqnarray}
\hat{\gamma}_i&=&\frac{1}{\sqrt{n}}\int\mathcal{M}_A(C_{m,n})d\mathbf{u}\nonumber\\
           &=&\frac{1}{\sqrt{n}}\int_0^1 \int_0^1  C_{m,n}(u_1,1,\ldots,1,u_{k+1}, 1, \ldots, 1)   du_1du_{k+1}\nonumber\\
           &=&\frac{1}{n}\left[\sum_{t=1}^{n-i}(\tilde{r}_t-1/2)(\tilde{r}_{t+k}-1/2) \right].
\label{rankco}             
\end{eqnarray}
\noindent and $\sqrt{n}\hat{\gamma}_i$ is asymptotically normal with mean zero and variance 
$1/12^{2}$. 
The same result carries over to the case of the squared residuals as discussed in Bouhaddioui and Ghoudi (2012). 
\end{proof}

\section*{References}

\noindent{\sc Andrews, B., Calder, M., and Davis, R. A.} (2009)
\newblock Maximum likelihood estimation for $\alpha$-stable autoregressive processes. 
\newblock {\em Ann. Statist.}, 37, 1946-1982.

\noindent{\sc Andrews, B. and Davis, R. A.} (2011) 
\newblock Model identification for infinite variance autoregressive processes.
\newblock{\em Annals of the Journal of Econometrics}.

\noindent{\sc Bouhaddioui, C. and Ghoudi, K.} (2012) 
\newblock Empirical processes for infinite variance autoregressive models. 
\newblock{\em Journal of Multivariate Analysis}, 107, 319-335.

\noindent{\sc Box, G. E. P. and Pierce, D. A.} (1970) 
\newblock Distribution of residual autocorrelations in autoregressive-integrated moving average time series models.
\newblock {\em J. Amer. Statist. Assoc.}, 65, 1509-1526.

\noindent{\sc Breidt, F. J. and Davis, R. A.} (1992) 
\newblock Time-reversibility, identifiability and independence of innovations for stationary time series.
\newblock {\em J. Time Ser. Anal.}, 13, 377-390. 

\noindent{\sc Breidt, F. J., Davis, R. A., Lii, K.-S., and Rosenblatt, M.} (1991) 
\newblock Maximum likelihood estimation for noncausal autoregressive processes.
\newblock{\em J. Multivariate Anal.}, 36, 175-198. 

\noindent {\sc Brockwell, P.~J.~\& Davis, R.~A.} (1991). 
\newblock {\em Time Series: Theory and Methods}, 2nd ed. New York: Springer.

\noindent {Cline, D. B. H. and Brockwell, P.} (1985) 
\newblock Linear prediction of ARMA processes with infinite variance.
\newblock {\em  Stochastic Process and Their Applications}, 19, 281-296. 

\noindent {Fisher, T. J. and Gallagher, C. M.} (2012)  
\newblock New Weighted Portmanteau Statistics for Time
Series Goodness-of-Fit Testing.
\newblock {\em Journal of the American Statistical Association},
accepted (2012).

\noindent {Gallagher, C. M.} (2001) 
\newblock A method for fitting stable autoregressive models using the autocovariation
function.
\newblock {\em Statist. Probab. Lett.}, 53, 381-390. 

\noindent {Genest, C. and R{\'e}millard, B.} (2004)
\newblock Tests of independence and randomness based on the
empirical copula process.
\newblock {\em Test} , 13, 335-370. 

\noindent {Lanne, M., Luoto, J., and Saikkonen, P.} (2010)
\newblock Optimal Forecasting of Noncausal Autoregressive
Time Series. 
\newblock {\em HEER (Helsinki Center of Economic Research)}, 
Discussion Paper No. 286. 

\noindent {Lee, S. and Ng, C. T.} (2010) 
\newblock Trimmed portmanteau test for linear processes with infinite
variance. 
\newblock {\em J. Multivariate Anal.}, 101, 984-998.

\noindent {Lin, J.-W. andMcLeod, A. I.} (2006) 
\newblock Improved Pe{\~n}a-Rodr{\'{\i}}guez portmanteau test. 
\newblock {\em Comput. Statist. Data Anal.}, 51, 1731-1738. 

\noindent {Ling, S.} (2005) 
\newblock Self-weighted least absolute deviation estimation for infinite variance autoregressive
models.
\newblock{\em J. R. Stat. Soc. Ser. B Stat. Methodol.}, 67, 381-393.

\noindent {Ljung, G. M. and Box, G. E. P.} (1978)
\newblock On a measure of lack of fit in time series models.
\newblock{\em Biometrika}, 65, 297-303.

\noindent {Mahdi, E. and McLeod, I. A.} (2012) 
\newblock Improved multivariate portmanteau test.
\newblock{\em Journal of Time Series Analysis}, 33, 211-222. 

\noindent {Monti, A. C.} (1994) 
\newblock A proposal for a residual autocorrelation test in linear models.
\newblock{\em Biometrika}, 81, 776-780. 

\noindent { Pe{\~n}a, D. and Rodr{\'{\i}}guez, J.} (2002)
\newblock A powerful portmanteau test of lack of fit for time series.
\newblock {\em J. Amer. Statist. Assoc.}, 97, 601-610. 

\noindent { Pe{\~n}a, D. and Rodr{\'{\i}}guez, J.} (2006)
\newblock The log of the determinant of the autocorrelation matrix for testing goodness of
fit in time series.
\newblock {\em J. Statist. Plann. Inference}, 136, 2706-2718.

\noindent {Rachev, S., Huber, I., Ortobelli, S., and Stoyanov, S.} (2004) 
\newblock Portfolio choice with heavy tailed distributions. 
\newblock {\em Technical report},\\ 
\newblock {\em URL http://www.pstat.ucsb.edu/research/papers/article3.pdf}

\noindent {Resnick, S. I.} (1997)  
\newblock Heavy tail modeling and teletraffic data.
\newblock {\em  Ann. Statist.}, 25, 1805-1869. 

\noindent {Rosenblatt, M.} (2000)
\newblock {\em Gaussian and non-Gaussian linear time series and random fields.}
New York:Springer-Verlag.

\noindent {Sheng, H. and Chen, Y.} (2011) 
\newblock FARIMA with stable innovations model of Great Salt Lake elevation time series.
\newblock {\em Signal Processing}, 91, 553-561. 

\noindent {Stuck, B. W. and Kleiner, B.} (1974) 
\newblock A statistical analysis of telephone noise.
\newblock {\em The Bell System Technical Journal} , 53, 1263-1320.

\noindent {Tokat, Y. and Schwartz, E. S.} (2002) 
\newblock The impact of fat tailed returns on asset allocation.
\newblock {\em Math. Methods Oper. Res.}, 55, 165-185.






\end{document}